%% file: ReggeSymm-arxiv.tex
\documentclass[a4paper,11pt]{amsart}
\usepackage[utf8]{inputenc}

\usepackage{amsmath,amssymb,amsthm}
\usepackage{graphicx,color}
\usepackage{url}
\usepackage{tikz-cd}
\graphicspath{{}}

\usepackage{anysize}
\marginsize{2.2cm}{2.2cm}{2cm}{2cm}

\newtheorem{thm}{Theorem}[section]
\newtheorem{lem}[thm]{Lemma}

\theoremstyle{definition}

\theoremstyle{remark}
\newtheorem{rem}[thm]{Remark}

\def\R{\mathbb R}
\def\E{\mathbb E}
\def\H{\mathbb H}
\def\Sph{\mathbb S}

\def\const{\mathrm{const}}

\def\Vol{\operatorname{Vol}}

\title[The Regge symmetry, confocal conics, and the Schl\"afli formula]{The Regge symmetry, confocal conics,\\ and the Schl\"afli formula}
\author{Arseniy Akopyan}
\address{Arseniy Akopyan, Institute of Science and Technology Austria (IST Austria), Am Campus~1, 3400 Klosterneuburg, AUSTRIA}
\email{akopjan@gmail.com}
\thanks{A.A. was supported by the European Research Council (ERC) under the European Union's Horizon 2020 research and innovation programme (grant agreement No 78818 Alpha).}

\author{Ivan Izmestiev}
\address{Ivan Izmestiev. University of Fribourg Department of Mathematics\\ Chemin du Mus\'ee 23 \\ CH-1700 Fribourg P\'erolles \\ SWITZERLAND}
\email{ivan.izmestiev@unifr.ch}
\thanks{I.I. was supported by the Swiss National Science Foundation grants 200021\_169391 and 200021\_179133.}

\begin{document}

\begin{abstract}
The Regge symmetry is a set of remarkable relations between two tetrahedra whose edge lengths are related in a simple fashion.
It was first discovered as a consequence of an asymptotic formula in mathematical physics.
Here we give a simple geometric proof of Regge symmetries in Euclidean, spherical, and hyperbolic geometry.
\end{abstract}

\maketitle

\section{Introduction}
The goal of this article is to give an elementary proof of the following theorem known as the \emph{Regge symmetry}.

\begin{thm}
\label{thm:Regge}
Let $\Delta$ be a spherical, hyperbolic, or Euclidean tetrahedron with edge lengths $x$, $y$, $a$, $b$, $c$, $d$
as shown in Figure \ref{fig:TwoTetra}, left.
Then there is a (respectively spherical, hyperbolic, or Euclidean) tetrahedron $\bar{\Delta}$ with edge lengths $x$, $y$, $s-a$, $s-b$, $s-c$, $s-d$,
where $s = \frac{a+b+c+d}2$, see Figure \ref{fig:TwoTetra}, right.
Besides, the following holds.
\begin{enumerate}
\item
The dihedral angles at the $x$-edge in $\Delta$ and $\bar{\Delta}$ are equal.
The same holds for the dihedral angles at the $y$-edge.
\item
If $\alpha$, $\beta$, $\gamma$, $\delta$ are the dihedral angles at the edges $a$, $b$, $c$, $d$ of $\Delta$,
then the dihedral angles at the edges $s-a$, $s-b$, $s-c$, $s-d$ in $\bar{\Delta}$ are equal to
$\sigma - \alpha$, $\sigma - \beta$, $\sigma - \gamma$, $\sigma - \delta$, where $\sigma = \frac{\alpha+\beta+\gamma+\delta}2$.
\item
Tetrahedra $\Delta$ and $\bar{\Delta}$ have equal volume.
\end{enumerate}
\end{thm}

\begin{figure}[ht]
\begin{center}
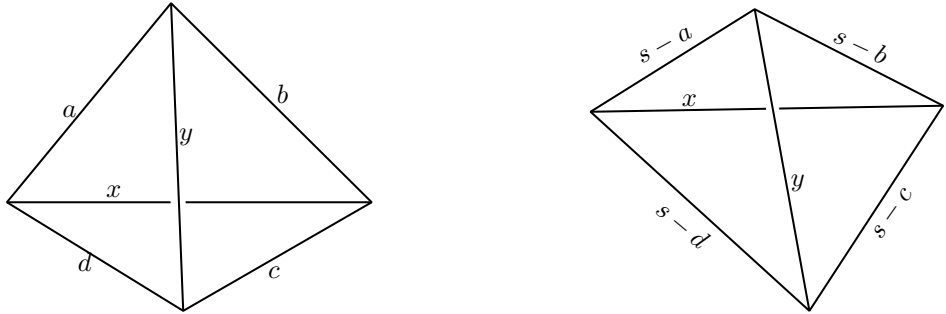
\end{center}
\caption{The Regge symmetry.}
\label{fig:TwoTetra}
\end{figure}

There are more relations between the metric elements of the tetrahedra $\Delta$ and $\bar{\Delta}$:
\begin{itemize}
\item the solid angle at the $(x, a, d)$ vertex is equal to the solid angle at the $(x, s-b, s-c)$ vertex etc.;
\item the logarithms of tangents of half-angles at the side $x$ are subject to the same linear transformations as $a, b, c, d$
and $\alpha, \beta, \gamma, \delta$.
\end{itemize}
The first of these follows directly from parts (1) and (2) of Theorem \ref{thm:Regge}.
The second relation is Lemma \ref{lem:MoreSymmetry} below.

The Euclidean case of Theorem \ref{thm:Regge} was first stated by Ponzano and Regge in \cite{PR}
as a corollary of a symmetry and a conjectured asymptotics of $6j$-symbols.
This asymptotic formula was proved by Roberts in \cite{Rob99};
similar formulas involving spherical and hyperbolic tetrahedra were obtained in \cite{TW05}.
Of course, one would like to have a more direct and geometric proof of Theorem \ref{thm:Regge}.

A brute-force proof of the Euclidean case was given in the Appendices B and D of \cite{PR}.
Roberts observed that the Dehn invariants of $\Delta$ and $\bar{\Delta}$ coincide.
In the Euclidean case Sydler's theorem \cite{Syd65} implies that $\Delta$ and $\bar{\Delta}$ are scissors congruent.
In fact, it would be interesting to prove Theorem \ref{thm:Regge} by exhibiting a scissors congruence between $\Delta$ and $\bar{\Delta}$
such that the sides and dihedral angles fit in an appropriate way.

Scissors congruence of $\Delta$ and $\bar{\Delta}$ in the hyperbolic case was proved by Mohanty in \cite{Moh03}.
Mohanty extends the tetrahedra to ideal polyhedra, which are then cut into pieces,
therefore her construction does not apply in the Euclidean and spherical cases.
It is unknown whether hyperbolic and spherical analogs of Sydler's theorem are true.
Thus it remains an open question whether $\Delta$ and $\bar{\Delta}$ are scissors congruent in the spherical case.

The proof of Theorem \ref{thm:Regge} given in this article uses classical geometric theorems
such that Ivory's lemma about confocal quadrics and Schl\"afli's differential formula for the volume of a tetrahedron.
There are some trigonometric computations, but they are kept to a minimum.

\section{Confocal conics}
\subsection{Ellipses and hyperbolas revisited}
In this section $a$, $b$, $c$ denote the side lengths of a triangle,
and $\alpha$, $\beta$, $\gamma$ the angles opposite to the corresponding sides.

\begin{lem}
\label{lem:ProdQuot}
\begin{enumerate}
\item
Let the side length $c$ in a spherical, hyperbolic, or Euclidean triangle be fixed.
Then the sum of the other two side lengths $a+b$
determines and is determined by the product of tangents of half-angles $\tan \frac{\alpha}2 \tan \frac{\beta}2$,
and the difference $a-b$ determines and is determined by the ratio $\frac{\tan\frac{\alpha}2}{\tan\frac{\beta}2}$.
\item
Let the angle $\gamma$ in a spherical or hyperbolic triangle be fixed.
Then the sum of the other two angles $\alpha+\beta$ determines and is determined
by the product of (hyperbolic) tangents of half-sides $\tan \frac{a}2 \tan \frac{b}2$ or $\tanh \frac{a}2 \tanh \frac{b}2$ respectively,
and the difference $\alpha-\beta$ determines and is determined by the ratio $\frac{\tan\frac{a}2}{\tan\frac{b}2}$
or $\frac{\tanh\frac{a}2}{\tanh\frac{b}2}$ respectively.
\end{enumerate}
\end{lem}

Part (1) of Lemma \ref{lem:ProdQuot} gives an alternative description of ellipses and hyperbolas:
instead of fixing the sum (or difference) of distances to the foci, one may fix the product (respectively quotient)
of the tangents of half-angles formed by the principal axis with the lines from the foci to a moving point, see Figure \ref{fig:ProdQuot}.

\begin{figure}[ht]
\begin{center}
\includegraphics{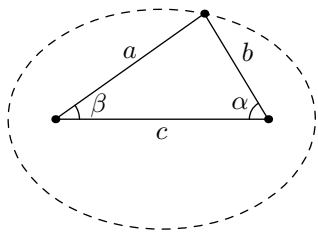}

\end{center}
\caption{$a+b = \const \Leftrightarrow \tan\frac{\alpha}2 \tan\frac{\beta}2 = \const$.}
\label{fig:ProdQuot}
\end{figure}

In the spherical case part (2) is dual to (and thus directly follows from) part (1).
In the hyperbolic case parts (1) and (2) become dual if they are extended to hyperbolic-de Sitter triangles.
Conics also appear in part (2): when the lines containing the sides $a$ and $b$ are fixed,
and the line containing the third side moves in such a way that the sum or difference $\alpha \pm \beta$ remains constant,
then the moving line remains tangent to a (spherical or hyperbolic) conic, see \cite{Izm17a}.

Finally note that part (2) of the Lemma can be reformulated so as to include the case of Euclidean triangles.
Let the angle $\gamma$ in a spherical, hyperbolic, or Euclidean triangle be fixed.
Then the area of the triangle determines and is determined by the products $\tan\frac{a}2 \tan\frac{b}2$,
$\tanh\frac{a}2 \tanh\frac{b}2$, and $ab$, respectively,
and the angle difference $\alpha - \beta$ determines and is determined by the quotients
$\frac{\tan\frac{a}2}{\tan\frac{b}2}$, $\frac{\tanh\frac{a}2}{\tanh\frac{b}2}$, and $\frac{a}{b}$, respectively.

The proof of Lemma \ref{lem:ProdQuot} is given in the last section of the paper.



\subsection{Confocal conics and quadrics}
For two points $F_1$ and $F_2$ in the Euclidean plane, the loci
\begin{equation}
\label{eqn:ConfocConics}
\{X \mid XF_1 + XF_2 = \const\} \text{ and } \{X \mid \left|XF_1 - XF_2\right| = \const\}
\end{equation}
form a \emph{confocal family} of conics: ellipses and hyperbolas with foci $F_1$ and $F_2$.
In a suitable coordinate system, the confocal family is described by the family of equations
\[
\frac{x^2}{a^2 - \lambda} + \frac{y^2}{b^2 - \lambda} = 1
\]
with parameter $\lambda$.

A spherical or hyperbolic conic is defined as the intersection of a quadratic cone with the unit sphere,
respectively with the hyperboloid in the Minkowski space modelling the hyperbolic plane.
Spherical and hyperbolic conics share many properties with Euclidean conics,
in particular they can be described by the same equations \eqref{eqn:ConfocConics}
(some of the foci of a hyperbolic conic may be ideal or hyperideal).
For a suitable choice of coordinates in $\R^3 \supset \Sph^2$ or $\R^{2,1} \supset \H^2$,
a confocal family of spherical or hyperbolic conics \eqref{eqn:ConfocConics} is described by the equations
\[
\frac{x^2}{a^2 - \lambda} + \frac{y^2}{b^2 - \lambda} - \frac{z^2}{c^2 \pm \lambda} = 0,
\]
where the $+$ sign is used in the spherical, and the $-$ sign in the hyperbolic case.
In other words, a confocal family is dual to a pencil of conics containing the absolute $x^2 + y^2 \pm z^2 = 0$.
For more details, see a recent survey \cite{Izm17a}.
Figure \ref{fig:ConfocConics} shows families of confocal conics on the sphere and in the Poincar\'e model of the hyperbolic plane.

\begin{figure}[ht]
\begin{center}
\includegraphics[width=.45\textwidth]{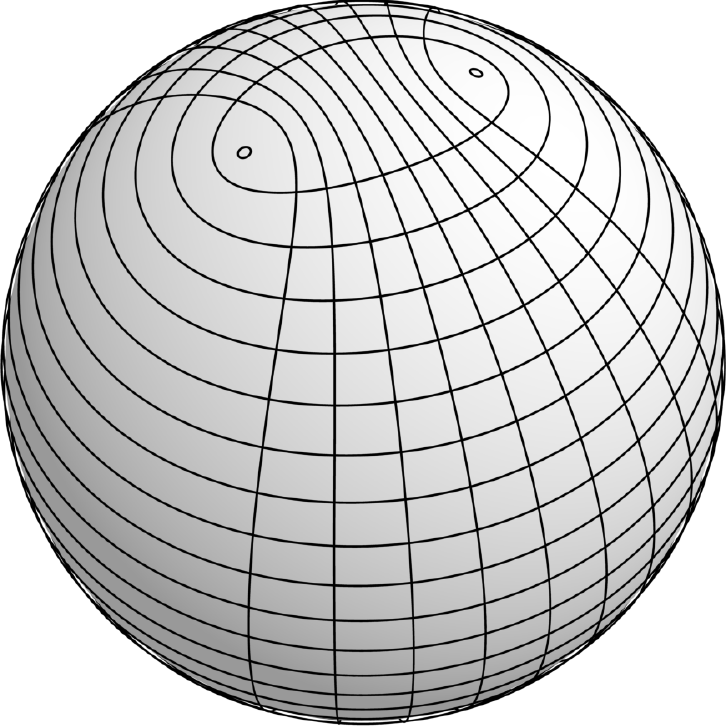}
\,\,\,\,
\includegraphics[width=.45\textwidth]{hyperbolic-conics-1.mps}
\end{center}
\caption{Confocal spherical and hyperbolic conics}
\label{fig:ConfocConics}
\end{figure}


The following theorem gives another description of confocal families, which follows directly from part (1) of Lemma \ref{lem:ProdQuot}.
\begin{thm}
\label{thm:EllCoord}
Let $F_1$ and $F_2$ be two points in $\E^2$ or $\Sph^2$ or $\H^2$.
For any point $X$ distinct from $F_1$ and $F_2$ denote
\[
k_1 = \tan \frac{\angle XF_1F_2}{2}, \quad k_2 = \tan \frac{\angle XF_2F_1}{2}.
\]

Then the conics with foci $F_1$ and $F_2$ are described by the equations
\[
k_1k_2 = \const \text{ and }\frac{k_1}{k_2} = \const.
\]
\end{thm}

\subsection{Ivory's lemma}
In higher dimensions, a confocal family of Euclidean quadrics is described by the equations
\begin{equation}
\label{eqn:ConfocQuad}
\frac{x_1^2}{a_1^2 - \lambda} + \frac{x_2^2}{a_2^2 - \lambda} + \cdots + \frac{x_n^2}{a_n^2 - \lambda} = 1, \quad a_1^2 > \cdots > a_n^2.
\end{equation}
The parameter $\lambda$ varies from $-\infty$ to $a_1^2$.
All quadrics with $\lambda < a_n^2$ are ellipsoids.
As $\lambda$ passes through $a_i^2$, the type of the quadric changes.
Quadrics of different type intersect orthogonally, see e.~g. \cite{Tab05}.

Choose two quadrics of each type, that is choose pairs of distinct numbers
\[
\lambda_1, \lambda'_1 \in (-\infty, a_n^2), \quad \lambda_2, \lambda'_2 \in (a_n^2, a_{n-1}^2), \quad \ldots, \quad \lambda_n, \lambda'_n \in (a_2^2, a_1^2).
\]
If we denote by $Q(\lambda)$ the quadric \eqref{eqn:ConfocConics},
then the quadrics $Q(\lambda_i)$ and $Q(\lambda'_i)$ bound a layer homeomorphic to $Q(\lambda_i) \times [0,1]$.
The intersection of these layers is a right-angled box bounded by quadrics.
Denote this box by $B(\Lambda)$, where $\Lambda = (\lambda_1, \lambda'_1, \ldots, \lambda_n, \lambda'_n)$ are the chosen parameter values.

Similarly, in the spherical or hyperbolical space, confocal quadrics are intersections of the cones
\[
\frac{x_1^2}{a_1^2 - \lambda} + \frac{x_2^2}{a_2^2 - \lambda} + \cdots + \frac{x_n^2}{a_n^2 - \lambda} - \frac{x_{n+1}^2}{a_{n+1}^2 \pm \lambda} = 0,
\quad a_1^2 > \cdots > a_n^2
\]
with the unit sphere, respectively with the standard hyperboloid of two sheets.
As above, one chooses $n$ pairs of quadrics of different types; they bound a box $B(\Lambda)$ in the spherical or hyperbolic space.

\begin{thm}[Ivory's lemma]
\label{thm:Ivory}
For any box $B(\Lambda)$ in any confocal family of quadrics in $\E^n$, $\Sph^n$, or $\H^n$ the following are true:
\begin{enumerate}
\item
The great diagonals of $B(\Lambda)$ are equal.
\item
Let $D_i \colon \R^n \to \R^n$ (in the spherical and hyperbolic cases $D_i \colon \R^{n+1} \to \R^{n+1}$),
be the affine transformation with a diagonal matrix which sends the quadric $Q(\lambda)$ to the quadric $Q(\lambda'_i)$.
Then $D_i$ sends the vertices of $B(\Lambda)$ situated on $Q(\lambda_i)$ to the vertices of $B(\Lambda)$ situated on $Q(\lambda'_i)$.
\end{enumerate}
\end{thm}

Figure \ref{fig:Ivory} illustrates the case $n=2$.
 
\begin{figure}[ht]
\begin{center}
\includegraphics{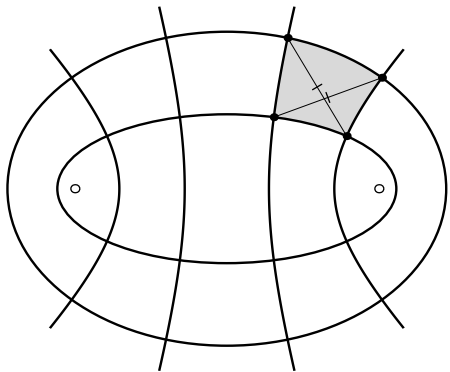}
\end{center}
\caption{Ivory's lemma in the plane.}
\label{fig:Ivory}
\end{figure}

Ivory's lemma holds also for degenerate confocal families obtained from non-generic pencils.
An example of such a family is obtained by rotating the family \eqref{eqn:ConfocConics} (see Figure \ref{fig:Ivory} and Figure \ref{fig:DihAngles}, top)
about the axis $F_1F_2$ and adding planes passing through this axis.

The original statement of Ivory's lemma (implicit in Ivory's work and made explicit by Chasles) deals with ellipsoids in $\E^3$.
The spherical and hyperbolic cases are discussed in \cite{StW04,IT17} but are special cases of a far more general theorem:
Ivory's lemma holds for boxes bounded by the coordinate hypersurfaces of a St\"ackel net, \cite{Bla28,Thimm78,IT17}.

\section{Dihedral angles}
\begin{proof}[Proof of Theorem \ref{thm:Regge}, part (1).]

Denote the vertices of the tetrahedron $\Delta$ by $F_1, F_2, K, L$ so that
\[
|F_1F_2| = x, \quad |F_1K| = a, \quad |F_2K| = b.
\]
In the plane of the triangle $F_1F_2K$ choose a point $\bar{L}$ so that the triangle $F_1F_2\bar{L}$
lies on the same side of $F_1F_2$ as the triangle $F_1F_2K$ and has side lengths
\[
|F_1 \bar{L}| = s-c, \quad |F_2 \bar{L}| = s-d.
\]
Similarly, in the plane $F_1F_2L$ construct a triangle $F_1F_2\bar{K}$ on the same side of $F_1F_2$ as $F_1F_2L$
with side lengths
\[
|F_1 \bar{K}| = s-b, \quad |F_2 \bar{K}| = s-a,
\]
see Figure \ref{fig:DihAngles}.
The segments $KL$ and $\bar{K}\bar{L}$ are great diagonals of a ``parallelopiped''
bounded by a degenerate confocal family of quadrics.
Indeed, due to $(s-c) + (s-d) = a + b$ and $(s-b) + (s-a) = c + d$
each of the pairs of points $K$, $\bar{L}$ and $\bar{K}$, $L$ lies on an ellipse with the foci $F_1$ and $F_2$.
And due to $(s-b) - (s-a) = a-b$ and $(s-c) - (s-d) = d-c$
each of the pairs $K$, $\bar{K}$ and $L$, $\bar{L}$ lies on a two-sheeted hyperboloid
obtained by rotating a hyperbola with foci $F_1$ and $F_2$ about the focal axis.
Thus, by Ivory's lemma one has $|\bar{K}\bar{L}| = |KL| = y$.

\begin{figure}[ht]
\begin{center}
\includegraphics[width=.6\textwidth]{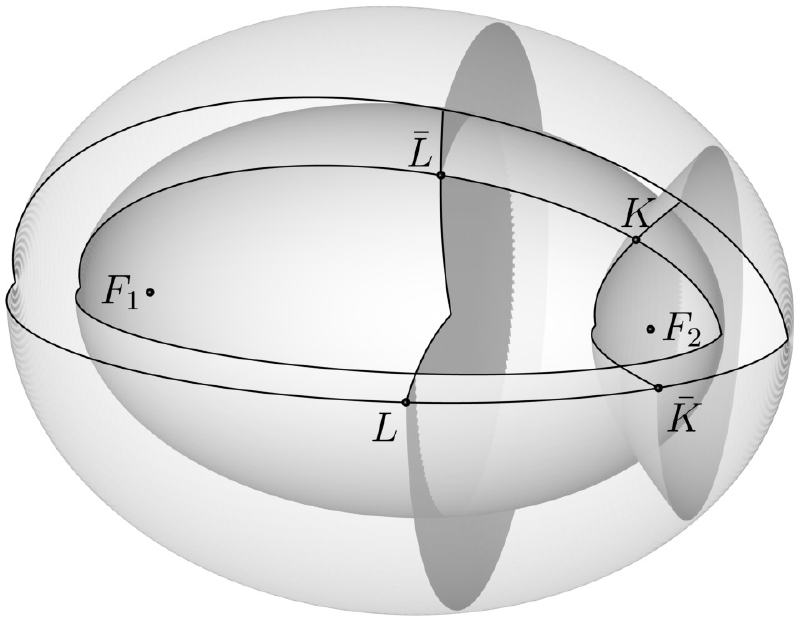}\\
\hspace{.6cm}\includegraphics[width=.47\textwidth]{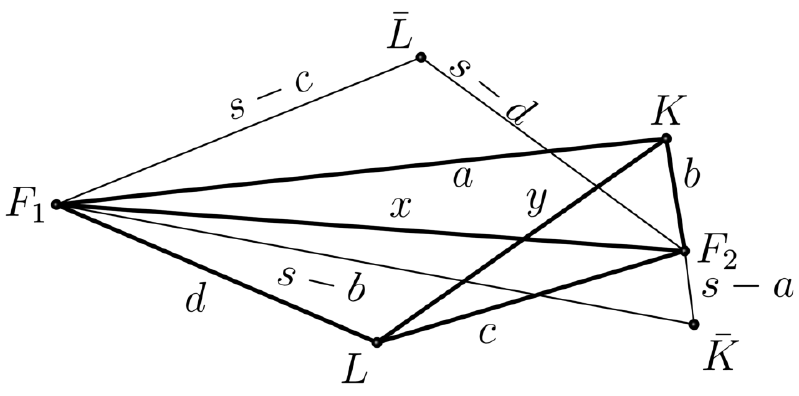}
\end{center}
\caption{To the proof of Theorem \ref{thm:Regge}, part (1).}
\label{fig:DihAngles}
\end{figure}

It follows that a tetrahedron $\bar{\Delta}$ exists and is isometric to the tetrahedron $F_1F_2\bar{K}\bar{L}$.
By construction, the dihedral angle of $\Delta$ at the $x$-edge is equal to the dihedral angle of $\bar{\Delta}$
at the $x$-edge.
By exchanging $x$ and $y$, one proves the equality of the dihedral angles at the $y$-edges.
\end{proof}

\begin{lem}
\label{lem:MoreSymmetry}
Denote by $\angle(a,x)$ the angle between the $a$-edge and $x$-edge of the tetrahedron $\Delta$ and let
\begin{gather*}
A = \log\tan \frac{\angle(a,x)}2, \quad B = \log\tan\frac{\angle(b,x)}2,\\
C = \log\tan\frac{\angle(c,x)}2, \quad D = \log\tan\frac{\angle(d,x)}2.
\end{gather*}
Similarly, let $\bar{A}$, $\bar{B}$, $\bar{C}$, $\bar{D}$ be the logarithms of the tangents of the half-angles
between the sides $s-a$, $s-b$, $s-c$, $s-d$ and the side $x$ of the tetrahedron $\bar{\Delta}$.
Then one has
\[
\bar{A} = \frac{-A+B+C+D}2 \text{ etc.}
\]
\end{lem}
\begin{proof}
Consider the triangles $(x,s-a,s-b)$ and $(x,c,d)$.
One has $(s-a) + (s-b) = c+d$.
Therefore by part (1) of Lemma \ref{lem:ProdQuot} the products of tangents of half-angles adjacent to the side $x$ in these triangles are equal.
In other words, one has $\bar{A} + \bar{B} = C + D$.

Similarly, in the triangles $(x,s-a,s-b)$ and $(x,a,b)$ one has $(s-a) - (s-b) = b-a$.
Therefore by part (1) of Lemma \ref{lem:ProdQuot} one has $\bar{A} - \bar{B} = B - A$.

The formula for $\bar{A}$ follows; formulas for $\bar{B}$, $\bar{C}$, $\bar{D}$ are proved similarly.
\end{proof}

%

\begin{proof}[Proof of Theorem \ref{thm:Regge}, part (2).]
Denote the dihedral angles in the tetrahedron $\bar{\Delta}$ at the edges $s-a$, $s-b$, $s-c$, $s-d$
by $\bar{\alpha}$, $\bar{\beta}$, $\bar{\gamma}$, and $\bar{\delta}$.

Consider the spherical section of the tetrahedron $\Delta$ at the vertex adjacent to the sides $x$, $b$, and $c$
and the spherical section of $\bar\Delta$ at the vertex adjacent to $x$, $s-a$, $s-d$.
These spherical triangles have a common angle: the dihedral angle at the edge $x$.
Besides, by Lemma \ref{lem:MoreSymmetry} one has $\bar A + \bar D = B + C$,
which implies that the products of tangents of half-sides adjacent to the common angle in these triangles are equal.
Thus by part (2) of Lemma \ref{lem:ProdQuot} one has $\bar\alpha + \bar\delta = \beta + \gamma$.

Apply Lemmas \ref{lem:MoreSymmetry} and \ref{lem:ProdQuot} to the spherical sections
at the vertices $(x,a,d)$ and $(x,s-a,s-d)$ to obtain $\bar\alpha - \bar\delta = \delta - \alpha$.

It follows that $\bar\alpha = \frac{-\alpha+\beta+\gamma+\delta}2 = \sigma - \alpha$
and similarly for the other dihedral angles.
\end{proof}

%
%
%
%

\section{Volume}
Let $A_1$ and $A_2$ be the areas of two faces of a Euclidean tetrahedron $\Delta$,
let $\ell_{12}$ be the length of their common edge, and let $\theta_{12}$ be the dihedral angle at this edge.
Then one has
\[
\Vol(\Delta) = \frac{2}{3} \frac{A_1 A_2 \sin\theta_{12}}{\ell_{12}}.
\]
This formula easily follows from the ``height times base'' formulas for the area of a triangle and the volume of a tetrahedron.
We now use it in combination with the Ivory lemma in order to establish the equality of volumes of $\Delta$ and $\bar{\Delta}$ in the Euclidean case.

\begin{proof}[Proof of Theorem \ref{thm:Regge}, part (3), Euclidean case]
Tetrahedra $\Delta$ and $\bar{\Delta}$ are shown on Figure \ref{fig:DihAngles}, bottom, as $F_1F_2KL$ and $F_1F_2K'L'$.
The above volume formula implies
\[
\Vol(\Delta) = \frac{2}{3} \frac{A_K A_L \sin\phi}{x}, \quad \Vol(\Delta') = \frac{2}{3} \frac{A_{\bar{K}} A_{\bar{L}} \sin\phi}{x},
\]
where $\phi$ is the angle of both $\Delta$ and $\bar{\Delta}$ at the edge $x$.
Thus it suffices to show that $A_K A_L = A_{\bar{K}} A_{\bar{L}}$ or, since all four triangles share a common side,
that $h_K h_L = h_{\bar{K}} h_{\bar{L}}$, where $h_K$ denotes the distance from $K$ to the line $F_1F_2$ etc.

By Theorem \ref{thm:Ivory} (applied to a degenerate family of confocal conics)
the affine transformation $(x,y,z) \mapsto (px, qy, qz)$ that sends the smaller ellipsoid to the bigger one
maps $K$ to a point $K'$ lying on the same hyperboloid.
The points $K'$ and $\bar{K}$ are related by the rotation by the angle $\phi$ about the axis $F_1 F_2$, see Figure \ref{fig:TetraPar}.

\begin{figure}[ht]
\begin{center}
\includegraphics[width=.5\textwidth]{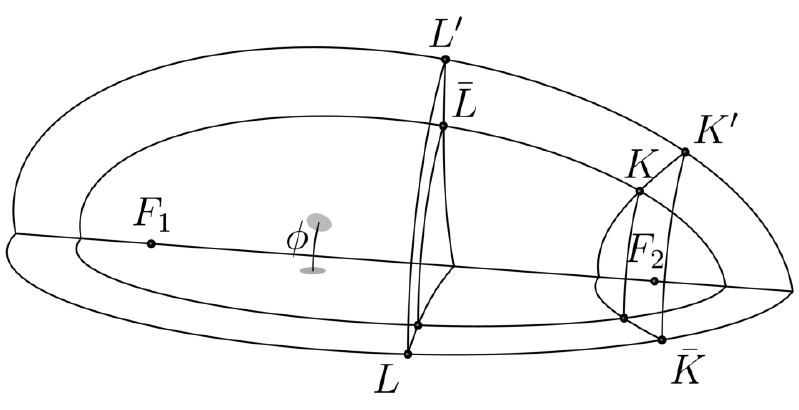}
\end{center}
\caption{$\Vol(\Delta) = \Vol(\bar{\Delta})$ in the Euclidean case.}
\label{fig:TetraPar}
\end{figure}

It follows that
\[
h_{\bar{K}} = h_{K'} = q h_K.
\]
Similarly, if $L'$ is the image of $\bar{L}$ under the above affine transformation, then one has
\[
h_{L} = h_{L'} = q h_{\bar{L}}.
\]
This implies $h_K h_L = h_{\bar K} h_{\bar L}$, and the equality of volumes is proved.
\end{proof}

\begin{rem}
The equality of volumes in the Euclidean case can also be proved by a straightforward calculation using the Cayley--Menger formula:
\[
\Vol(\Delta)^2 =
\frac{1}{288}
\begin{vmatrix}
0 & 1 & 1 & 1 & 1\\
1 & 0 & a^2 & b^2 & y^2\\
1 & a^2 & 0 & x^2 & d^2\\
1 & b^2 & x^2 & 0 & c^2\\
1 & y^2 & d^2 & c^2 & 0
\end{vmatrix}
\]
\end{rem}

Volumes of spherical and hyperbolic tetrahedra is a very different story from the volume of a Euclidean tetrahedron:
formulas for the volume are much more complicated, see a recent survey \cite{AM19}.
In order to prove the equality of volumes in $\Sph^3$ or $\H^3$, we use
the following theorem (\cite{Schl60}, for modern expositions see \cite{AVS93, Mil94}).
\begin{thm}[Schl\"afli]
For every smooth deformation of a spherical or hyperbolic tetrahedron one has
\[
d\Vol = \pm \frac12 \sum_{i=1}^6 \ell_i d\theta_i,
\]
where the sum is taken over the edges of the tetrahedron, $\ell_i$ is the length of the $i$-th side,
and $\theta_i$ is the dihedral angle at the $i$-th side.
The sign on the right hand side is $+$ in the spherical and $-$ in the hyperbolic case.
\end{thm}

For a deformation of a Euclidean tetrahedron one has $\sum_{i=1}^6 \ell_i d\theta_i = 0$,
but we do not need this formula.

\begin{proof}[Proof of Theorem \ref{thm:Regge}, part (3), spherical and hyperbolic cases]
Deform the tetrahedron $\Delta$ by increasing the length of the side $y$ until the tetrahedron flattens
(the angle at the $x$-edge becomes equal to $\pi$).
Let $\Delta_t$ be the tetrahedron with $y=t$,
and let $\bar{\Delta_t}$ be the tetrahedron corresponding to~$\Delta_t$ under the Regge symmetry.
Since the dihedral angles of $\Delta_t$ and $\bar{\Delta_t}$ at the $x$-edges are equal,
at $t = y_{max}$ both tetrahedra flatten.
Thus we have
\[
\Vol(\Delta_t) \big|_{t = y_{\max}} = 0 = \Vol(\bar{\Delta_t}) \big|_{t = y_{max}}.
\]
By the Schl\"afli formula one has
\begin{gather*}
\frac{d}{dt} \Vol(\Delta_t) = \pm \frac12 (a \dot\alpha + b \dot\beta + c \dot\gamma + d \dot\delta + x \dot\phi + t \dot\psi)\\
\frac{d}{dt} \Vol(\bar{\Delta_t}) =
\pm \frac12 \left( (s-a)(\dot\sigma - \dot\alpha) + \cdots + (s-d)(\dot\sigma - \dot\delta) + x \dot\phi + t \dot\psi \right)
\end{gather*}
A simple computation
\begin{multline*}
(s-a)(\dot\sigma - \dot\alpha) + \cdots + (s-d)(\dot\sigma - \dot\delta)\\
= 4s \dot\sigma - (a+b+c+d) \dot\sigma - s (\dot\alpha + \dot\beta + \dot\gamma + \dot\delta)
+ a\dot\alpha + b\dot\beta + c\dot\gamma + d\dot\delta\\
= a\dot\alpha + b\dot\beta + c\dot\gamma + d\dot\delta
\end{multline*}
shows that the derivatives of the volumes coincide for all $t$.
It follows that $\Vol(\Delta_t) = \Vol(\bar{\Delta_t})$ for all $t$, and in particular $\Vol(\Delta) = \Vol(\bar{\Delta})$.
\end{proof}

\begin{rem}
Note that the Euclidean case can be derived from the spherical or hyperbolic cases
by taking the limit over a sequence of spherical tetrahedra shrinking to a point and tending in their shape to~$\Delta$.
\end{rem}

\section{Proof of Lemma~\ref{lem:ProdQuot}}
The lemma can be proved by synthetic methods, especially in the Euclidean case.
For the hyperbolic and the spherical cases one uses the Poincar\'e model or the stereographic projection.

Since synthetic proofs are somewhat lengthy, here we provide a short analytical proof.

The cosine law expresses an angle of a triangle in terms of its side lengths.
In the spherical and hyperbolic case there are dual cosine laws which express a side in terms of the angles.
By using classical trigonometric identities one can give these expressions a different form, known as half-angle and half-side formulas
(see e.g. \cite{Tod86} for the proof of the spherical case).
Below $s = \frac{a+b+c}2$ is the semiperimeter of the triangle.

\begin{itemize}
\item
For a spherical triangle:
\[
\tan \frac{\alpha}2 = \sqrt{\frac{\sin(s-b) \sin(s-c)}{\sin s \sin(s-a)}}, \quad
\tan \frac{a}{2} = \sqrt{\frac{-\cos\sigma \cos(\sigma-\alpha)}{\cos(\sigma-\beta) \cos(\sigma-\gamma)}}.
\]
\item
For a hyperbolic triangle:
\[
\tan \frac{\alpha}2 = \sqrt{\frac{\sinh(s-b) \sinh(s-c)}{\sinh s \sinh(s-a)}}, \quad
\tanh \frac{a}{2} = \sqrt{\frac{\cos\sigma \cos(\sigma-\alpha)}{\cos(\sigma-\beta) \cos(\sigma-\gamma)}}.
\]
\item
For a Euclidean triangle:
\[
\tan \frac{\alpha}2 = \sqrt{\frac{(s-b)(s-c)}{s(s-a)}}.
\]
\end{itemize}

Multiplying two half-angle formulas for a spherical triangle one obtains
\[
\tan\frac{\alpha}2 \tan\frac{\beta}2 = \frac{\sin(s-c)}{\sin s}
= \frac{\sin\left( \frac{a+b}2 - \frac{c}2 \right)}{\sin\left( \frac{a+b}2 + \frac{c}2 \right)}
= \frac{\tan\frac{a+b}2 - \tan\frac{c}2}{\tan\frac{a+b}2 + \tan\frac{c}2}.
\]
Thus, for a given $c$, the value of $\tan\frac{\alpha}2 \tan\frac{\beta}2$ is uniquely determined by the value of $a+b$.
The equation can be solved for $\tan\frac{a+b}2$, so that $a+b$ is uniquely determined by $\tan\frac{\alpha}2 \tan\frac{\beta}2$.

Similarly, dividing two half-angle formulas for a spherical triange one obtains
\[
\frac{\tan\frac{\alpha}2}{\tan\frac{\beta}2} = \frac{\sin(s-b)}{\sin(s-a)}
= \frac{\sin\left(\frac{c}2 + \frac{a-b}2\right)}{\sin\left(\frac{c}2 - \frac{a-b}2\right)}
= \frac{\tan\frac{c}2 + \tan\frac{a-b}2}{\tan\frac{c}2 - \tan\frac{a-b}2}.
\]

Computations are similar in the hyperbolic case, as well as for the products and ratios of tangents of half-sides.

In the Euclidean case by mutliplying and dividing the half-angle formulas one gets
\[
\tan\frac{\alpha}2 \tan\frac{\beta}2 = \frac{a+b-c}{a+b+c}, \quad
\frac{\tan\frac{\alpha}2}{\tan\frac{\beta}2} = \frac{c+(a-b)}{c-(a-b)}.
\]



\bibliographystyle{abbrv}
\bibliography{../ReggeSymm}

\end{document}

%% file: TwoTetra.pdf_t
\begin{picture}(0,0)%
\includegraphics{TwoTetra.pdf}%
\end{picture}%
\setlength{\unitlength}{3315sp}%
\begingroup\makeatletter\ifx\SetFigFont\undefined%
\gdef\SetFigFont#1#2#3#4#5{%
  \reset@font\fontsize{#1}{#2pt}%
  \fontfamily{#3}\fontseries{#4}\fontshape{#5}%
  \selectfont}%
\fi\endgroup%
\begin{picture}(6974,2339)(69,-1613)
\put(826,-742){\makebox(0,0)[lb]{\smash{{\SetFigFont{10}{12.0}{\rmdefault}{\mddefault}{\updefault}{\color[rgb]{0,0,0}$x$}%
}}}}
\put(1367,-316){\makebox(0,0)[lb]{\smash{{\SetFigFont{10}{12.0}{\rmdefault}{\mddefault}{\updefault}{\color[rgb]{0,0,0}$y$}%
}}}}
\put(502,-147){\makebox(0,0)[lb]{\smash{{\SetFigFont{10}{12.0}{\rmdefault}{\mddefault}{\updefault}{\color[rgb]{0,0,0}$a$}%
}}}}
\put(2090,-40){\makebox(0,0)[lb]{\smash{{\SetFigFont{10}{12.0}{\rmdefault}{\mddefault}{\updefault}{\color[rgb]{0,0,0}$b$}%
}}}}
\put(2024,-1345){\makebox(0,0)[lb]{\smash{{\SetFigFont{10}{12.0}{\rmdefault}{\mddefault}{\updefault}{\color[rgb]{0,0,0}$c$}%
}}}}
\put(617,-1292){\makebox(0,0)[lb]{\smash{{\SetFigFont{10}{12.0}{\rmdefault}{\mddefault}{\updefault}{\color[rgb]{0,0,0}$d$}%
}}}}
\put(5086,-59){\makebox(0,0)[lb]{\smash{{\SetFigFont{10}{12.0}{\rmdefault}{\mddefault}{\updefault}{\color[rgb]{0,0,0}$x$}%
}}}}
\put(5893,-658){\makebox(0,0)[lb]{\smash{{\SetFigFont{10}{12.0}{\rmdefault}{\mddefault}{\updefault}{\color[rgb]{0,0,0}$y$}%
}}}}
\put(4882,-856){\rotatebox{320.0}{\makebox(0,0)[lb]{\smash{{\SetFigFont{10}{12.0}{\rmdefault}{\mddefault}{\updefault}{\color[rgb]{0,0,0}$s-d$}%
}}}}}
\put(6196,421){\rotatebox{335.0}{\makebox(0,0)[lb]{\smash{{\SetFigFont{10}{12.0}{\rmdefault}{\mddefault}{\updefault}{\color[rgb]{0,0,0}$s-b$}%
}}}}}
\put(4792,248){\rotatebox{29.0}{\makebox(0,0)[lb]{\smash{{\SetFigFont{10}{12.0}{\rmdefault}{\mddefault}{\updefault}{\color[rgb]{0,0,0}$s-a$}%
}}}}}
\put(6559,-1058){\rotatebox{55.0}{\makebox(0,0)[lb]{\smash{{\SetFigFont{10}{12.0}{\rmdefault}{\mddefault}{\updefault}{\color[rgb]{0,0,0}$s-c$}%
}}}}}
\end{picture}%